\documentclass[11pt]{amsart}
\usepackage{amssymb, graphicx, color}
\usepackage{amsfonts}
\usepackage{amsthm}
\usepackage{enumerate}
\usepackage[mathscr]{eucal}
\usepackage{graphicx}

\usepackage[bookmarksnumbered, bookmarksopen, colorlinks, citecolor=blue, linkcolor=blue]{hyperref}

\allowdisplaybreaks

\newtheorem{Atheorem}{Theorem}

\newtheorem{Alemma}[Atheorem]{Lemma}

\newtheorem{Btheorem}{Theorem}
\newtheorem{Bproposition}[Btheorem]{Proposition}

\newtheorem{Bremark}[Btheorem]{Remark}

\newtheorem{thm}{Theorem}[section]

\newtheorem{cor}[thm]{Corollary}
\newtheorem{lem}[thm]{Lemma}

\newtheorem{prop}[thm]{Proposition}

\newtheorem{remark}[thm]{Remark}
\newtheorem{rk}[thm]{Remark}

\theoremstyle{definition}
\theoremstyle{remark}
\numberwithin{equation}{section}

\newcommand{\R}{{\mathbb R}}

\newcommand{\Z}{{\mathbb Z}}

\newcommand{\M}{{\mathcal M}}

\newcommand{\bs}{\begin{split}}
\newcommand{\es}{\end{split}}

\newcommand{\be}{\begin{eqnarray*}}
\newcommand{\ee}{\end{eqnarray*}}
\newcommand{\beq}{\begin{align}}
\newcommand{\eeq}{\end{align}}

\def\Q{\mathcal{Q}}
\def\1{\mathbf{1}}

\begin{document}

%
%
%
%
%
%
%
%
\setcounter{page}{1}
\title[Noncommutative differential transforms]
{Noncommutative differential transforms for averaging operators}


\author[B. Xu]{Bang Xu}

\address{School of Mathematics and Statistics, Wuhan University,  Wuhan 430072, China; and
	Department of Mathematical Sciences, Seoul National University, Seoul 08826, Republic of Korea}

\email{bangxu@whu.edu.cn}



\date{January 22, 2022.
}
\begin{abstract}
In this paper, we complete the study of mapping properties for a family of operators evaluating the difference between differentiation operators and conditional expectations acting on noncommutative $L_{p}$-spaces.
To be more precise, we establish the weak type $(1,1)$ and $(L_{\infty},\mathrm{BMO})$ estimates of this difference. Consequently, in conjunction with interpolation and duality, we obtain all strong type $(p,p)$ estimates. This allows us to obtain a quick application to noncommutative differential transforms for averaging operators.
\end{abstract}

\maketitle

\section{Introduction}
In the scalar-valued case, Jones et al (see \cite{JRW03,HM1}) investigated the mapping properties of the following square function
\begin{align}\label{square}
\big(\sum_{k}|(M_{k}-\mathsf{E}_{k})f(x)|^{2}\big)^{\frac{1}{2}}
\end{align}
defined for a reasonable function $f$. Given $k\in\Z$, $M_{k}$ denotes the Hardy-Littlewood averaging operator
$$M_k f(x)= \frac{1}{|B_k|} \int_{B_k}f(x+y)dy= \frac{1}{|B_k|} \int_{\R^d}f(y) \1_{B_k}(x-y)~dy \quad x\in\R^d,$$
where $B_k$ is the open ball centered at the origin $0$ with radius $r(B_k)$ equals to $2^{-k}$.
$\mathsf{E}_{k}$ denotes the $k$-th conditional expectation with respect to the $\sigma$-algebra  generated by the standard dyadic cubes whose side-length equal to $2^{-k}$. It is worth pointing out that such square function plays a crucial role in deducing the variational inequalities for ergodic averages. The variational inequalities are much stronger than the maximal inequalities and thus have been widely explored in ergodic theory, probability and harmonic analysis. For topics on variational theory, we refer the reader to, for instance, \cite{Bou89,CJRW00,HM1,CJRW03,JKRW98,JR02,JRW03,JSW08,JoWa04,LeXu2,Mas,MaTo12,MaTo,OSTTW12}.

In \cite{HX}, the author and his collaborator extended the mapping properties of the square function defined as (\ref{square}) to the operator-valued setting. More precisely, let $\M$ be a
semifinite von Neumann algebra equipped with a normal semifinite
faithful (abbrieviated as \emph{n.s.f}) trace $\tau$ and $\mathcal N=L_{\infty}(\R^{d})\overline{\otimes}\M$ be the tensor von Neumann algebra with the \emph{n.s.f} tensor trace $\varphi=\int dx\otimes \tau$, where $dx$ is the Lebesgue measure. For $0<p\leq\infty$, we write $L_p(\M)$ and $L_p(\mathcal N)$ for the noncommutative $L_p$-spaces associated to the pairs $(\M,\tau)$ and $(\mathcal N,\varphi)$.
Note that for $0<p<\infty$, the second noncommutative $L_p$-space $L_p(\mathcal N)$ can be identified as the Bochner $L_{p}$-space
$L_{p}(\R^{d};L_p(\M))$.
Our previous work \cite{HX} established the weak type $(1,1)$, strong type $(p,p)$ ($1<p<\infty$) and  ($L_{\infty},\mathrm{BMO}$) bounds for the square function operator defined in (\ref{square}) when acting on operator-valued functions. Equivalently, 
by noncommutative Khintchine inequalities in $L_{1,\infty}$ \cite{C1} and in $L_{p}$ \cite{LG}, 
we were reduced to showing the weak type $(1,1)$ and strong type $(p,p)$ estimates of the following linear operator
\begin{align}\label{finite1}
{L}f(x)=\sum_{k}\varepsilon_{k}(M_{k}-\mathsf{E}_k)f(x),
\end{align}
where $(\varepsilon_{k})$ is a Rademacher sequence on a probability space $(\Omega,P)$. 

\medskip

Motivated by the boundedness theory of $L$ defined as (\ref{finite1}), we would like to investigate the mapping properties of the following
linear operator
\begin{align}\label{1}
Tf(x)=\sum_{k}\nu_{k}(M_{k}-\mathsf{E}_k)f(x),
\end{align}
whenever the sequence $(\nu_{k})\in\ell_{\infty}$ and $f\in C_{c}^{\infty}(\R^{d})\otimes S_{\M}$, where $S_{\M}$ is the subset of $\M$ with $\tau$-finite support.

We will establish the following result and we refer the reader to Section 2 (resp. Section 4) for the definition of weak $L_{p}$-spaces $L_{p,\infty}(\mathcal N)$ (resp. $\mathrm{BMO}_{d}(\mathcal N)$ spaces).
\begin{thm}\label{t5}\rm
Let $1\leq p\leq\infty$ and $T$ be defined as (\ref{1}). Let $\nu=(\nu_{k})$ be any sequence such that $\|\nu\|_{\ell_{\infty}}\leq1$. Then the following assertions hold with a positive constant $C_{p,d}$ depending on $p$ and the dimension
$d$:
\begin{itemize}
\item[(i)] for $p=1$, $$\|Tf\|_{L_{1,\infty}(\mathcal N)}\leq C_{p,d}\|f\|_{1},\; \forall f\in L_{1}(\mathcal N);$$

\item[(ii)] for $1<p<\infty$, $$\| T \! f\|_{p}\leq C_{p,d}\|f\|_{p},\; \forall f\in L_{p}(\mathcal N);$$

\item[(iii)] for $p=\infty$, $$\|
T f \|_{\mathrm{BMO}_{d}(\mathcal N)}\leq C_{p,d} \,
\|f\|_\infty,\; \forall f\in L_{\infty}(\mathcal N).$$
\end{itemize}
\end{thm}
\begin{remark}\label{app}
\emph{The three estimates announced in Theorem \ref{t5} for infinite summations  over $k\in\mathbb Z$ should be understood as a consequence of the corresponding uniform boundedness for all finite summations with the standard approximation arguments (see e.g. \cite[Section 6.A]{JMX}). Therefore, as in \cite{MP,HX}, we will not explain the convergence of infinite sums appearing in the rest of the paper if there is no ambiguity.}
\end{remark}
\begin{remark}
\emph{In particular, when $(\nu_{k})$ is replaced by the Rademacher sequence $(\varepsilon_{k})$, we can immediately obtain the weak type $(1,1)$ and strong type $(p,p)$ results stated in \cite[Theorem 1.1]{HX} thanks to the noncommutative Khintchine inequalities in $L_{1,\infty}$-space \cite{C1} and in $L_{p}$-spaces \cite{LG}.
}
\end{remark}
If we set
\begin{align}\label{1211}
Df(x)=\sum_{k}\nu_{k}\big(M_k f(x)-M_{k-1}f(x)\big),
\end{align}
then together with the noncommutative martingale transforms \cite{Ran02}, Theorem \ref{t5} finds its second important application.
\begin{cor}\label{t7}\rm
Let $D$ be the differential transform defined as (\ref{1211}) with $\|(\nu_{k})\|_{\ell_{\infty}}\leq1$. Then the following estimates hold:
\begin{itemize}
\item[(i)] for $p=1$ and $f\in L_{1}(\mathcal N)$, $$\|Df\|_{L_{1,\infty}(\mathcal{N})}\leq C_{d}\|f\|_{1};$$

\item[(ii)] for $1<p<\infty$ and $f\in L_{p}(\mathcal N)$, $$\|Df\|_{p}\leq C_{p,d} \|f\|_{p}.$$
\end{itemize}
\end{cor}

\medskip

Let us briefly analyse the proof of Theorem \ref{t5}.
For the $(L_{\infty},\mathrm{BMO})$ estimate, the proof is standard and all the strong type $(p,p)$ estimates can be obtained by duality and interpolation since the strong type $(2,2)$ estimate of $T$ follows trivially from the commutative result and we prefer to present a noncommutative proof in Appendix A for warming up. Thus we give our main efforts to the weak type (1,1) estimate.

However, with a moment's thought, it is difficult to adapt the argument in \cite{JP1,MP,C,HX} to our setting. It is clear that the kernel associated with ${T}$ does not enjoy any regularity while the methods applied in \cite{JP1,MP,C} depend heavily on Lipschitz's regularity condition (see also \cite{HLX,CCP} for the weaker regularity condition).
On the other hand, the almost orthogonality principle used in \cite[Lemma 3.7]{HX} is no longer applicable for the present case. The main ingredient is a new kind of noncommutative Calder\'on-Zygmund decomposition communicated to us by Cadilhac \cite{C2}, see Theorem \ref{czdecom}. 
This new CZ decomposition circumvents the pseudo-localization technique \cite{JP1} and still yields $L_1$ endpoint of standard CZ operators for operator-valued functions (see \cite{CCP} for more details). Furthermore,
to make this new Calder\'on-Zygmund decomposition work in our case, we are partially motivated by Cadilhac's note \cite{C2}. 

The rest of the paper is organized as follows. In Section 2, we review some standard facts on noncommutative $L_{p}$-spaces. This section also introduces the content of noncommutative Calder\'on-Zygmund decomposition. In Section 3, we prove the weak type $(1,1)$ and strong type $(p,p)$ results announced in Theorem \ref{t5}, as well as Corollary
\ref{t7}. Section 4 deals with the $(L_{\infty},\mathrm{BMO})$ estimate in Theorem \ref{t5}. The proof of $L_{2}$ boundedness of $T$ will be presented in Appendix A. Finally, we include Appendix B with some further problems.

\textbf{Notation:} In all what follows, we write
$X\lesssim Y$ if $X\le CY$ for some inessential constant $C>0$ and we write $X\thickapprox Y$ to mean that $X\lesssim Y$ and $Y \lesssim X$.
\section{Preliminaries}
\subsection{Noncommutative $L_{p}$-spaces}\quad

\medskip

Let $\M$ be a semifinite von Neumann algebra equipped with a \emph{n.s.f} trace $\tau$. Let $\M_{+}$ be the positive part of $\M$ and denote by $\mathcal{S_{\M+}}$ the set of all $x\in\M_{+}$ such that $\tau(\mathrm{supp}x)<\infty$, where $\mathrm{supp}x$ means the support of $x$. Let $\mathcal{S}_{\M}$ be the linear span of $\mathcal{S_{\M+}}$. Then $\mathcal{S}_{\M}$ is a $w^{*}$-dense $\ast$-subalgebra of $\M$. Given $0< p<\infty$, we set
$$\|x\|_{p}=[\tau(|x|^p)]^{1/p}\ \ x\in\mathcal{S}_{\M},$$
where $|x|=(x^{\ast}x)^{\frac{1}{2}}$ is the modulus of $x$, it turns out that $\|\cdot\|_{p}$ is a norm on $\mathcal{S}_{\M}$ if $1\leq p<\infty$, and a $p$-norm if $0< p<1$. The completion of $(\mathcal{S}_{\M},\|\cdot\|_{p})$ is denoted by $L_{p}(\M,\tau)$ or simply by $L_{p}(\M)$. As usual, we set $L_{\infty}(\M) = \M$ equipped with the operator norm $\|\cdot\|_{\M}$. The  positive part of $L_{p}(\M)$ is written as $L_{p}(\M)_{+}$.

Suppose that $\M\subset B(\mathcal{H})$ acts on a separable Hilbert space $\mathcal{H}$. A closed densely defined operator on $\mathcal{H}$ is said to be affiliated with $\M$ if it commutes with any unitary of $\M'$, where $\M'$ is the commutant of $\M$. If $x$ is a densely defined selfadjoint operator on $\mathcal{H}$
and $x = \int_{\R} \lambda \hskip1pt d \gamma_x(\lambda)$ is its spectral
decomposition, the spectral projection $\int_{\mathcal{I}} d
\gamma_x(\lambda)$ will be simply denoted by $\chi_{\mathcal{I}}(x)$, where $\mathcal{I}$ is a measurable subset of $\R$. A closed and densely defined operator
$x$ affiliated with $\mathcal{M}$ is called \emph{$\tau$-measurable} if
there is $\lambda > 0$ such that $$\tau \big( \chi_{(\lambda,\infty)}
(|x|) \big) < \infty.$$
Let $L_{0}(\M)$ be the family of the $\ast$-algebra of \emph{$\tau$-measurable} operators. For $0< p<\infty$, the weak $L_{p}$-space $L_{p,\infty}(\M)$ is defined as the set of all $x$ in $L_0(\M)$ with the following finite quasi-norm
$$\|x\|_{p,\infty}=\sup_{\lambda > 0}\lambda\tau \big( \chi_{(\lambda,\infty)}
(|x|) \big)^{\frac{1}{p}}.$$
It is already shown in \cite[Lemma 2.1]{JRWZ} that for any $x_1, x_2 \in
L_{1,\infty}(\M)$ and any $\lambda\in\R_{+}$
\begin{align}\label{distri}
\tau\big(\chi_{(\lambda,\infty)}
(|x_1+x_2|)\big)\leq \tau\big(\chi_{(\lambda/2,\infty)}
(|x_1|)\big)+\tau\big(\chi_{(\lambda/2,\infty)}(|x_2|)\big).
\end{align}
The reader is referred to \cite{FK,P2} for more information on noncommutative $L_p$-spaces.
\subsection{Noncommutative Hilbert-valued $L_{p}$-spaces}\quad

\medskip

In this subsection, we briefly introduce noncommutative Hilbert-valued $L_{p}$-spaces.
Let $(\Sigma,\mu)$ be a measurable space. The column space $L_p(\M;L^c_2(\Sigma))$ is the family of the operator-valued functions $f$ with finite norm (resp. $p$-norm) for $p\geq1$ (resp. $0<p<1$)
$$\|f\|_{L_p(\M;L^c_2(\Sigma))}=\Big\|\Big(\int_{\Sigma}f^*(\omega)f(\omega)d\mu(\omega)\Big)^{\frac12}\Big\|_{p}.$$
The reader is referred to \cite{P2} for precise definition and related properties of the noncommutative Hilbert-valued $L_p$-spaces.
The main property for our purpose is the following H\"older type inequality (see e.g. \cite[Proposition 1.1]{M}).
\begin{lem}\label{mainlemma}\rm
Let $0<p,q,r\leq\infty$ be such that $1/r=1/p+1/q$. Then for any $f\in L_p(\M;L^c_2(\Sigma))$ and $g\in L_q(\M;L^c_{2}(\Sigma))$
$$\Big\|\int_{\Sigma}f^*(\omega)g(\omega)d\mu(\omega)\Big\|_{r}\leq \Big\|\Big(\int_{\Sigma}|f(\omega)|^2d\mu(\omega)\Big)^{\frac12}\Big\|_{p}
\Big\|\Big(\int_{\Sigma}|g(\omega)|^2d\mu(\omega)\Big)^{\frac12}\|_{q}.$$
\end{lem}
\subsection{Noncommutative dyadic martingales}\quad

\medskip

Recall that a dyadic cube in $\R^{d}$ is the set
$$[2^{k}m_{1},2^{k}(m_{1}+1))\times\cdot\cdot\cdot\times[2^{k}m_{d},2^{k}(m_{d}+1)),$$
where $k,m_{1},\cdot\cdot\cdot,m_{d}\in\Z$.
Denote by $\Q$ the set of all dyadic cubes in $\R^d$.
Given an integer $k \in \Z$, let $\Q_k$ be the set of dyadic cubes with side length $\ell(Q)=2^{-k}$. Write $\sigma_k$ for the $\sigma$-algebra generated by $\Q_k$ and $\mathcal N_k=L_\infty(\mathbb R^d,\sigma_k,dx)\overline{\otimes}\M$ for the associated von Neumann subalgebra of $\mathcal N$, where $\mathcal N=L_\infty(\mathbb R^d)\overline{\otimes}\M$ was given in the introduction. Then $(\mathcal{N}_k)_{k \in \Z}$ is a sequence of increasing von Neumann subalgebras whose union $\cup_{k}\mathcal{N}_k$ is weak$^*$ dense in $\mathcal N$, and whence form a filtration with the associated dyadic conditional expectations $(\mathsf{E}_k)_{k\in\Z}$ defined as
$$\mathsf{E}_k(f):=f_{k}=\sum_{Q \in \Q_k}^{\null} f_Q \chi_Q,\;\forall f\in L_1(\mathcal N)$$
where $\chi_Q$ is the characteristic function of $Q$ and 
$$f_Q = \frac{1}{|Q|} \int_Q f(y) \, dy.$$
Here $|Q|$ is the Lebesgue measure of $Q$.

Let $1\leq p\leq\infty$. The sequence $(f_k)_{k\in\mathbb Z}$ is called the $L_p$-martingale if for any integer $k$, 
$f_k\in L_p(\mathcal N)$.  The martingale difference sequence $df=(df_{k})_{k\in\Z}$ of $f$ is defined by $df_{k}=f_{k}-f_{k-1}$.


\subsection{Noncommutative Calder{\'o}n-Zygmund decomposition}\quad

\medskip

We end this section with a description of the noncommutative Calder{\'o}n-Zygmund decomposition \cite{CCP}, whose construction is
based on Cuculescu's theorem \cite{Cuc}. 


As in \cite{C,JP1,HLX}, the noncommutative Calder{\'o}n-Zygmund decomposition will be constructed for operator-valued functions  in the following class
$$\mathcal N_{c,+} =\Big\{
f: \R^d \to \M\cap L_1(\M) \, \big| \ f \geq0, \
\overrightarrow{\mathrm{supp}} \hskip1pt f \ \ \mathrm{is \
compact} \Big\},$$
which is dense in $L_1(\mathcal N)_{+}$. Here
$\overrightarrow{\mathrm{supp}}$ stands for the support of $f$ as an
operator-valued function in $\R^d$, which is different from its support projection in $\mathcal N$. Moreover, given $f \in \mathcal{N}_{c,+}$ and $\lambda>0$, there exists $m_{\lambda}(f)\in\Z$ such that $f_{k}\leq\lambda\1_{\mathcal{N}}$ for all $k\leq m_{\lambda}(f)$ (see \cite[Lemma 3.1]{JP1}), where $\1_{\mathcal{N}}$ stands for the unit
elements in $\mathcal{N}$.

The following modified Cuculescu's theorem \cite{Cuc} can be found in \cite[Lemma 3.1]{JP1}.
\begin{lem}\label{cucu}\rm
Let $f \in \mathcal{N}_{c,+}$ and consider the associated positive
$L_{1}$-martingale $(f_k)_{k\in\mathbb Z}$  relative to the filtration $(\mathcal{N}_k)_{k \in \Z}$. Given $\lambda>0$,  there exists a
sequence of decreasing projections $(q_k)_{k\in\Z}$ defined by $q_k = \1_\mathcal{N}$ for $k\leq m_{\lambda}(f)$ and recursively for $k>m_{\lambda}(f)$
$$q_k=q_k(f,\lambda)=\chi_{(0,\lambda]}(q_{k-1} f_k q_{k-1})$$
such that
\begin{itemize}
\item [(i)] $q_k$ commutes with $q_{k-1} f_k
q_{k-1}$ for each $k\in\Z$;

\item [(ii)] $q_k$ belongs to $\mathcal{N}_k$ and $q_k f_k q_k \le \lambda \hskip1pt
q_k$ for each $k\in\Z$;

\item [(iii)] the following estimate holds $$\varphi \Big(
\mathbf{1}_\mathcal{N} - \bigwedge_{k \in\Z} q_k \Big) \le
\frac{\|f\|_1}{\lambda}.$$
\end{itemize}
\end{lem}
In fact, in the present operator-valued setting, there exists another useful expression for the $q_k$'s constructed in Lemma \ref{cucu}. It is easy to verify that
$$q_{k}=\sum_{Q\in\Q_{k}}q_{Q}\chi_{Q},$$
for $k\in\Z$, and $q_{Q}$ are projections in $\M$ with
$$q_{Q}=\begin{cases} \1_\M & \mbox{if} \ k \leq
m_{\lambda}(f),
\\ \chi_{(0,\lambda]} \big( q_{\widehat{Q}} f_Q q_{\widehat{Q}}\big) & \mbox{if} \ k > m_{\lambda}(f),
\end{cases}$$
where $\widehat{Q}$ is the dyadic father of $Q$. Accordingly these projections satisfy
\begin{equation}\label{czd5}
\  q_Q\leq q_{\widehat{Q}},\ \
\  q_Q\  \mbox{commutes\ with}\  q_{\widehat{Q}} f_Q q_{\widehat{Q}},\ \
\   q_Q f_Q q_Q \le \lambda q_Q.
\end{equation}

Moreover, if we define $p_k$ as
\begin{equation}\label{czd1}
p_k=q_{k-1}-q_k=\sum_{Q\in\Q_{k}}(q_{\widehat{Q}}-q_{Q})\chi_{Q}\triangleq\sum_{Q\in\Q_{k}}p_Q\chi_{Q},
\end{equation}
then it is not difficult to check that
$$\sum_{k \in \Z} p_k = \1_\mathcal{N} - q =q^\perp\quad \mbox{with} \quad q = \bigwedge_{k \in
\Z} q_k$$
and clearly for each $k\in\Z$,
\begin{equation}\label{czd1123}
\|p_kf_{k}p_k\|_{\infty}\leq2^{d}\lambda.
\end{equation}


The following theorem was communicated to us by Cadilhac \cite{C2} and we refer the reader to \cite{CCP} for more details.
\begin{thm}\label{czdecom}\rm
Fix $f\in\mathcal N_{c,+}$ and $\lambda>0$. Let $(q_k)_{k\in\Z}$ and $(p_k)_{k\in\Z}$ be two sequences of projections appeared above. Define a projection $\zeta\in \mathcal N$ as
\begin{equation}\label{czd9}
\zeta = \big(\bigvee_{Q\in \Q} p_Q\chi_{5Q}\big)^{\bot},
\end{equation}
where $5Q$ is the cube with the same center as $Q$ such that $\ell(5Q)=5\ell(Q)$.
Then there exists a decomposition of $f$,
\begin{equation}\label{czd76789}
f = g + b
\end{equation}
such that the following assertions hold.

\begin{enumerate}[(1)]
\item [(i)]$\varphi(\mathbf 1_\mathcal{N}-\zeta) \leq 5^{d}\dfrac{\|f\|_1}{\lambda}$.
\item[(ii)]$g=qfq+\sum_{n\in\Z}p_{n}f_{n}p_{n}$ satisfies:
$\| g\|_1 \le\|f\|_1 \quad \mbox{and} \quad \|
g\|_\infty \le 2^{d} \lambda$.


\item[(iii)]$b=\sum_{n\in\Z}b_{n}$, where
\begin{equation}\label{czd6}
b_{n}=p_n (f-f_{n}) q_{n}+q_{n-1}(f-f_{n})p_n.
\end{equation}
Each $b_{n}$ satisfies two cancellation conditions: for $Q\in \Q_{n}$,
     $\int_Q b_{n} = 0;$
    and for all $x,y\in\R^{d}$ such that $y \in 5Q_{x,n}$, $\zeta(x)b_{n}(y)\zeta(x) = 0$, where $Q_{x,n}$ stands for the unique cube in $\Q_n$ containing $x$.
\end{enumerate}
\end{thm}

\section{Proof of Theorem \ref{t5}: Weak type $(1,1)$ and strong type $(p,p)$ estimates}
In this section, we prove conclusions $(i)$ and $(ii)$ stated in Theorem \ref{t5}, while Corollary \ref{t7} will be shown at the end of this section. We begin with the weak type $(1,1)$ estimate.

By decomposing $f = f_{1}-f_{2} +i(f_{3}-f_{4})$ with positive $f_{j}$ such that $\|f_{j}\|_{1}\leq\|f\|_{1}$ for $j=1,2,3,4$, we may assume that $f$ is positive. Since $\mathcal N_{c,+}$ is dense in $L_1(\mathcal N)_{+}$ and by  the standard approximation argument, it suffices to show the desired weak type $(1,1)$ estimate for $f\in \mathcal N_{c,+}$. Now fix one $f\in \mathcal N_{c,+}$ and a $\lambda\in(0,+\infty)$. Without loss of generality, we may assume that $m_{\lambda}(f)=0$ in this subsection. By the noncommutative Calder\'on-Zygmund decomposition-Theorem \ref{czdecom} and the property of distribution (\ref{distri}),
$$\varphi(\chi_{(\lambda,\infty)}(|Tf|))\leq \varphi(\chi_{(\lambda/2,\infty)}(|Tg|))+\varphi(\chi_{(\lambda/2,\infty)}(|Tb|)).$$
Therefore, it suffices to show
$$
\varphi(\chi_{(\lambda/2,\infty)}(|Th|))\lesssim\frac{\| f\|_{1}}{\lambda}
$$
for $h=g$ and $b$.
\subsection{Weak type estimate for good function}\quad

\medskip

We will need the following fact that $T$ is bounded on $L_{2}(\mathcal N)$, whose proof can be found in Appendix A.
\begin{lem}\label{t1}\rm
Let $h\in L_{2}(\mathcal N)$. Then there exists a constant $C_{d}$ depending only on the dimension
$d$ such that
$$\|Th\|_{2}\leq C_{d}\|h\|_{2}.$$
\end{lem}
\begin{prop}\label{good function}\rm
The following estimate holds.
$$
\varphi(\chi_{(\lambda/2,\infty)}(|Tg|))\lesssim\frac{\| f\|_{1}}{\lambda}.
$$
\end{prop}
\begin{proof}
The desired weak type $(1,1)$ estimate for good function follows from Lemma \ref{t1} and conclusion (ii) in Theorem \ref{czdecom}. Indeed, by the Chebychev and H\"{o}lder inequalities,
$$\varphi(\chi_{(\lambda/2,\infty)}(|Tg|))\leq\frac{\|Tg\|^{2}_{2}}{\lambda^{2}}
\lesssim\frac{\|g\|_{2}^{2}}{\lambda^{2}}\leq\frac{\|g\|_{1}\|g\|_{\infty}}{\lambda^{2}}
\lesssim\frac{\|f\|_{1}}{\lambda}.$$
This completes the argument of our assertion for $Tg$.
\end{proof}
\subsection{Weak type estimate for the bad function}\quad

\medskip

Using the
projection $\zeta$ introduced in (\ref{czd9}), we
decompose $Tb$ in the following way
$$Tb = (\1_\mathcal{N} - \zeta) T
b (\1_\mathcal{N} - \zeta) + \zeta \hskip1pt T b (\1_\mathcal{N} - \zeta)
+ (\1_\mathcal{N} - \zeta) T b \zeta + \zeta \hskip1pt T b
\zeta.$$
Therefore, by (\ref{distri}) and conclusion (i) in Theorem \ref{czdecom},
\begin{align*}
\hskip1pt \varphi \big(\chi_{(\lambda/2,\infty)}(|Tb|)
\big)
\lesssim&\ \hskip1pt \varphi (\1_\mathcal{N} - \zeta) +
\hskip1pt \varphi \big(\chi_{(\lambda/8,\infty)}(|\zeta Tb\zeta|)
\big)\\
\lesssim&\ \frac{ \|f\|_{1}}{\lambda}+\varphi \big(\chi_{(\lambda/8,\infty)}(|\zeta Tb\zeta|)
\big).
\end{align*}
Hence, we are reduced to estimating $\varphi \big(\chi_{(\lambda/8,\infty)}(|\zeta Tb\zeta|)
\big)$.
\begin{prop}\label{bad function}\rm
The following estimate holds.
$$\varphi \big(\chi_{(\lambda/8,\infty)}(|\zeta Tb\zeta|)\big)\lesssim\frac{ \|f\|_{1}}{\lambda}.$$
\end{prop}
To prove Proposition \ref{bad function}, by noting that $m_{\lambda}(f)=0$, we have $p_n=0$ for all $n\leq 0$. Thus we can express $b$ as $b=\sum_{n=1}^{\infty}b_{n}$, where $b_{n}=p_n (f-f_{n}) q_n+q_{n-1}(f-f_{n})p_n$ as in (\ref{czd6}).
\begin{lem}\label{bad estimate}\rm
Fix $n\geq1$, we have 
\begin{align}\label{bad342212237289}
\sum_{k:k< n}\|M_{k}b_{n}\|_{1}\lesssim\Big(\lambda\varphi(p_{n})\Big)^{\frac{1}{2}}\Big(\varphi(fp_{n})\Big)^{\frac{1}{2}}.
\end{align}
\end{lem}
\begin{proof}
By the cancellation property of $b_{n}$ stated in Theorem \ref{czdecom} (iii), it is easy to verify that
\begin{align}\label{bad3422237289}
M_kb_{n}(x)=\frac{1}{|B_{k}|}\int_{x+B_{k}}b_{n}(y)dy=\frac{1}{|B_{k}|}\sum_{{\begin{subarray}{c}
Q\in\Q_{n} \\ (x+\partial B_{k})\cap Q\neq \emptyset
\end{subarray}}}\int_{(x+B_{k})\cap Q}b_{n}(y)dy,
\end{align}
where $\partial B$ means the boundary of the ball $B$. On the other hand, we claim that
\begin{align}\label{bad34237289}
b_{n}=p_n f q_n+q_{n-1}fp_n-q_{n-1}f_{n}p_n.
\end{align}
Indeed, by the definition of $p_{n}$-(\ref{czd1}), $p_{n}=q_{n-1}-q_{n}\leq q_{n-1}$. Hence, the conclusion (i) in Lemma \ref{cucu} gives
$$p_n f_{n} q_n=p_nq_{n-1} f_{n} q_{n-1}q_n=p_nq_nq_{n-1} f_{n} q_{n-1}=0.$$
This gives the desired expression.

In the following, we set
\begin{eqnarray*}
\|F_{1}\|_{1}& \triangleq &\sum_{k:k<n}\int_{\R^{d}}\Big\|\frac{1}{|B_{k}|}\sum_{{\begin{subarray}{c}
Q\in\Q_{n} \\ (x+\partial B_{k})\cap Q\neq \emptyset
\end{subarray}}}\int_{(x+B_{k})\cap Q}(p_n f q_n)(y)dy\Big\|_{L_{1}(\M)}dx,\\
\|F_{2}\|_{1}& \triangleq &\sum_{k:k<n}\int_{\R^{d}}\Big\|\frac{1}{|B_{k}|}\sum_{{\begin{subarray}{c}
Q\in\Q_{n} \\ (x+\partial B_{k})\cap Q\neq \emptyset
\end{subarray}}}\int_{(x+B_{k})\cap Q}(q_{n-1}fp_n)(y)dy\Big\|_{L_{1}(\M)}dx, \\
\|F_{3}\|_{1}& \triangleq &\sum_{k:k<n}\int_{\R^{d}}\Big\|\frac{1}{|B_{k}|}\sum_{{\begin{subarray}{c}
Q\in\Q_{n} \\ (x+\partial B_{k})\cap Q\neq \emptyset
\end{subarray}}}\int_{(x+B_{k})\cap Q}(q_{n-1}f_{n}p_n)(y)dy\Big\|_{L_{1}(\M)}dx.
\end{eqnarray*}
Then by observations (\ref{bad34237289}), (\ref{bad3422237289}) and the Minkowski inequality, we conclude that to obtain the desired inequality (\ref{bad342212237289}), it suffices to show
$$\|F\|_{1}\triangleq\max\Big\{\|F_{1}\|_{1},\|F_{2}\|_{1},\|F_{3}\|_{1}\Big\}\lesssim\Big(\lambda\varphi(p_{n})\Big)^{\frac{1}{2}}\Big(\varphi(fp_{n})\Big)^{\frac{1}{2}}.$$
Firstly, we use Lemma \ref{mainlemma} and (\ref{czd1}) to get 
\begin{align}\label{diff1}
\begin{split}
&\quad\Big\|\int_{(x+B_k)\cap Q}p_Q f(y) q_Qdy\Big\|_{L_{1}(\M)}\\
 & \leq\Big\|\Big(\int_{(x+B_k)\cap Q}p_Qf(x)p_Qdx\Big)^{\frac{1}{2}}\Big\|_{L_{1}(\M)}
\Big\|\Big(\int_{(x+B_k)\cap Q}q_Qf(x)q_Qdx\Big)^{\frac{1}{2}}\Big\|_{L_{\infty}(\M)}\\& \leq\Big\|\Big(\int_{ Q}p_Qf(x)p_Qdx\Big)^{\frac{1}{2}}\Big\|_{L_{1}(\M)}
\Big\|\Big(\int_{Q}q_Qf(x)q_Qdx\Big)^{\frac{1}{2}}\Big\|_{L_{\infty}(\M)}\\
& =\Big\|\int_{ Q}p_Qf(x)p_Qdx\Big\|^{\frac{1}{2}}_{L_{\frac{1}{2}}(\M)}
\Big\|\int_{Q}q_Qf(x)q_Qdx\Big\|^{\frac{1}{2}}_{L_{\infty}(\M)}.
\end{split}
\end{align}
Then applying the H\"{o}lder inequality to the first term and property (\ref{czd5}) to the second term above, we obtain
\begin{align*}
\hskip1pt \Big\|\int_{(x+B_k)\cap Q}p_Q f(y) q_Qdy\Big\|_{L_{1}(\M)}
\leq&\ \|p_{Q}\|^{\frac{1}{2}}_{L_{1}(\M)}
\Big\|\int_{Q}p_Qf(x)p_Qdx\Big\|^{\frac{1}{2}}_{L_{1}(\M)}(\lambda|Q|)^{\frac{1}{2}}\\
=&\ (\varphi(fp_{Q}\chi_{Q})\tau(p_{Q})\lambda|Q|)^{\frac{1}{2}}.
\end{align*}
The similar argument as in (\ref{diff1}) implies
\begin{align*}
\hskip1pt \Big\|\int_{(x+B_k)\cap Q}q_{\widehat{Q}}f(y)p_{Q}dy\Big\|_{L_{1}(\M)}
\leq&\ \Big\|\int_{ Q}p_Qf(x)p_Qdx\Big\|^{\frac{1}{2}}_{L_{\frac{1}{2}}(\M)}
\Big\|\int_{Q}q_{\widehat{Q}}f(x)q_{\widehat{Q}}dx\Big\|^{\frac{1}{2}}_{L_{\infty}(\M)}\\
\lesssim&\ (\varphi(fp_{Q}\chi_{Q})\tau(p_{Q})\lambda|Q|)^{\frac{1}{2}},
\end{align*}
where in the second inequality, we used the following basic fact $$\|q_{n-1}f_{n}q_{n-1}\|_{\infty}\leq2^{d}\|q_{n-1}f_{n-1}q_{n-1}\|_{\infty}\leq2^{d}\lambda.$$
Likewise, combining with the trace preserving property $\varphi(f_{n}p_{n})=\varphi(fp_{n})$, we see that
$$\Big\|\int_{(x+B_k)\cap Q}q_{\widehat{Q}}f_{n}(y)p_{Q}dy\Big\|_{L_{1}(\M)}
\lesssim(\varphi(fp_{Q}\chi_{Q})\tau(p_{Q})\lambda|Q|)^{\frac{1}{2}}.$$
Putting these estimates together, using the Minkowski inequality
and the Fubini theorem, we find that
\begin{align*}
\hskip1pt \|F\|_{1}
\lesssim&\ \hskip1pt \sum_{k:k<n}2^{kd}\int_{\R^{d}}\sum_{{\begin{subarray}{c}
Q\in\Q_{n} \\ (x+\partial B_{k})\cap Q\neq \emptyset
\end{subarray}}}(\varphi(fp_{Q}\chi_{Q})\tau(p_{Q})\lambda|Q|)^{\frac{1}{2}}dx\\
=&\ \sum_{k:k<n}2^{kd}\sum_{Q\in\Q_{n}}
(\varphi(fp_{Q}\chi_{Q})\tau(p_{Q})\lambda|Q|)^{\frac{1}{2}}\int_{\R^{d}}\chi_{(x+\partial B_{k})\cap Q\neq \emptyset}dx.
\end{align*}
Note that
$$\int_{\R^{d}}\chi_{(x+\partial B_{k})\cap Q\neq \emptyset}dx\lesssim2^{-k(d-1)}2^{-n}.$$
Then by above estimate, the Fubini theorem and the Cauchy-Schwarz inequality, we finally deduce that
\begin{align*}
\|F\|_{1}
\lesssim&\ \hskip1pt \sum_{k:k<n}\sum_{Q\in\Q_{n}}2^{k-n}(\varphi(fp_{Q}\chi_{Q})\tau(p_{Q})\lambda|Q|)^{\frac{1}{2}}\\
\lesssim&\ \sum_{Q\in\Q_{n}}(\varphi(fp_{Q}\chi_{Q})\tau(p_{Q})\lambda|Q|)^{\frac{1}{2}}\\
\leq&\ \Big(\sum_{
Q\in\Q_{n}}\tau(p_{Q})
|Q|\lambda\Big)^{\frac{1}{2}}\Big(\sum_{
Q\in\Q_{n}}\varphi(fp_{Q} \chi_{Q})\Big)^{\frac{1}{2}}\\
=&\ \Big(\lambda\varphi(p_{n})\Big)^{\frac{1}{2}}\Big(\varphi(fp_{n})\Big)^{\frac{1}{2}}.
\end{align*}
Thus we finish the proof of the lemma.
\end{proof}
Now we are ready to prove Proposition \ref{bad function}.
\begin{proof}[Proof of Proposition \ref{bad function}.]
Note that the Chebychev inequality implies that
$$\varphi \big(\chi_{(\lambda/8,\infty)}(|\zeta Tb\zeta|)\big)\lesssim\frac{\|\zeta Tb\zeta\|_{1}}{\lambda}.$$
Hence, it is enough to show
\begin{eqnarray}\label{9}
\|\zeta Tb\zeta\|_{1}\lesssim\|f\|_{1}.
\end{eqnarray}
To estimate (\ref{9}), applying the Minkowski inequality and since $\|(\nu_{k})\|_{\ell_\infty}\leq1$, we get
\begin{eqnarray*}\label{333}
\|\zeta Tb\zeta\|_{1}\leq
\sum_{n=1}^{\infty}\sum_{k}
\|\zeta\nu_{k}(M_k-\mathsf{E}_{k}) \hskip-1pt b_{n}\zeta
\|_{1}\lesssim\sum_{n=1}^{\infty}\sum_{k}
\|\zeta(M_k-\mathsf{E}_{k}) \hskip-1pt b_{n}\zeta
\|_{1}.
\end{eqnarray*}
In the following, we claim that
\begin{eqnarray}\label{diff}
\sum_{n=1}^{\infty}\sum_{k}\|\zeta(M_k-\mathsf{E}_{k}) \hskip-1pt b_{n}\zeta
\|_{1}\leq\sum_{n=1}^{\infty}\sum_{k:k<n}\|M_k b_{n}\|_{1}.
\end{eqnarray}
To see this, note that for  $k\geq n$, $\zeta(M_k-\mathsf{E}_{k}) b_{n}\zeta=0$. Indeed, fix one $x\in\mathbb R^d$. By the cancellation property announced in Theorem \ref{czdecom} (iii), one has
\begin{align}\label{bad34789}
\zeta(x)M_{k}b_{n}(x)\zeta(x)=\zeta(x)\frac1{|B_{k}|}\int_{x+B_{k}}b_n(y)\1_{y\notin 5Q_{x,n}}dy\zeta(x)=0,
\end{align}
since $x+B_{k}\subset 5Q_{x,n}$; similarly
$$
\zeta(x)\mathsf{E}_k b_{n}(x)\zeta(x)=\zeta(x)\frac{1}{|Q_{x,k}|} \int_{Q_{x,k}}b_{n}(y)\1_{y\notin 5Q_{x,n}}dy\zeta(x)=0,
$$
since $Q_{x,k}\subset Q_{x,n}$. On the other hand, for $k< n$, by applying the property of conditional expectation: $\mathsf{E}_k b_{n}=\mathsf{E}_k\mathsf{E}_nb_{n}=0$. Finally, since $\zeta$ is a projection in $\mathcal{N}$, we  establish  claim (\ref{diff}).

Therefore, by (\ref{diff}), we are reduced to showing
\begin{align}\label{bad347289}
\sum_{n=1}^{\infty}\sum_{k:k<n}\|M_k b_{n}\|_{1}\lesssim\|f\|_{1}.
\end{align}
However, the estimate (\ref{bad347289}) can be obtained by Lemma \ref{bad estimate}. Indeed,  using the Cauchy-Schwarz inequality, we have
\begin{align*}
  \sum_{n=1}^{\infty}\sum_{k:k<n}\|M_k b_{n}\|_{1}
  & \lesssim\Big(\sum_{n=1}^{\infty}\lambda\varphi(p_{n})\Big)^{\frac{1}{2}}
  \Big(\sum_{n=1}^{\infty}\varphi(fp_{n})\Big)^{\frac{1}{2}}\\
& =  \Big(\lambda\varphi(1-q)\Big)^{\frac{1}{2}}
  \Big(\varphi(f(1-q))\Big)^{\frac{1}{2}}\lesssim\|f\|_{1},
\end{align*}
where the last inequality follows from conclusion (iii) stated in Lemma \ref{cucu}.
This completes the argument for $Tb$.
\end{proof}
\begin{rk}\label{clever}\rm
What we need to point out is that the idea in proving Proposition \ref{bad function} is  partially inspired by Cadilhac's note \cite{C2}.
\end{rk}
\subsection{Conclusion}
Combining the weak type $(1,1)$ estimate of $T$ and Lemma \ref{t1}, we get the strong type $(p,p)$ $(1<p<2)$ estimate of $T$ by real interpolation \cite{P2}. The strong type $(p,p)$ $(2<p<\infty)$ estimate of $T$ can be deduced by duality. Therefore, we finish the proof of conclusion (i) and (ii) announced in Theorem \ref{t5}.

\bigskip

At the end of this section, we are at a position to prove Corollary \ref{t7}.
\begin{proof}[Proof of Corollary \ref{t7}.]
This is an immediate consequence of Theorem \ref{t5}. Indeed, 
\begin{align*}
Df(x)& =\sum_{k}\nu_{k}\Big((M_{k}-\mathsf{E}_{k})f(x)+(\mathsf{E}_{k}-\mathsf{E}_{k-1})f(x)+(\mathsf{E}_{k-1}-M_{k-1})f(x)\Big)\\
& =\sum_{k}\nu_{k}(M_{k}-\mathsf{E}_{k})f(x)+\sum_{k}\nu_{k}(\mathsf{E}_{k}-\mathsf{E}_{k-1})f(x)
+\sum_{k}\nu_{k}(\mathsf{E}_{k-1}-M_{k-1})f(x).
\end{align*}
The first and the third terms in the final expression above can be handled by conclusion (i) and (ii) in Theorem \ref{t5}, while the middle term we refer to \cite[Corollary 3.8]{Ran02} (see also \cite[Theorem 3.3.2]{P2}). This completes the proof of Corollary \ref{t7}.
\end{proof}
\section{Proof of Theorem \ref{t5}: $(L_{\infty},\mathrm{BMO})$ estimate}
In this section, we examine the $(L_{\infty},\mathrm{BMO})$ estimate announced in Theorem \ref{t5}. We first recall the definition of noncommutative dyadic $\mathrm{BMO}$ spaces. According to \cite{M}, the space $\mathrm{BMO}_{d}$ is defined as the closure of functions $f$ in $L_{1,loc}(\R^d;\M)$ with
$$\|f\|_{\mathrm{BMO}_{d}(\mathcal{N})} =\max \Big\{
\|f\|_{\mathrm{BMO}^{r}_{d}(\mathcal{N})}, \|f\|_{\mathrm{BMO}^{c}_{d}(\mathcal{N})}
\Big\}<\infty,$$ where the row and column dyadic $\mathrm{BMO}_{d}$
norms are given by
\begin{eqnarray*}
\|f\|_{\mathrm{BMO}^{r}_{d}(\mathcal{N})} & = & \sup_{Q \in \Q} \Big\| \Big(
\frac{1}{|Q|} \int_Q \Big|\big(f(x) -f_Q\big)^{\ast}\Big|^2 \, dx \Big)^{\frac12} \Big\|_{\M}, \\
\|f\|_{\mathrm{BMO}^{c}_{d}(\mathcal{N})} & = & \sup_{Q \in \Q} \Big\| \Big(
\frac{1}{|Q|} \int_Q \Big|f(x) -f_Q\Big|^2 \, dx \Big)^{\frac12} \Big\|_{\M}.
\end{eqnarray*}
From the definition of $\mathrm{BMO}$ spaces, we are reduced to showing
\begin{align}\label{Bmo1}
\|Tf\|_{\mathrm{BMO}^{r}_{d}(\mathcal{N})}\lesssim\|f\|_{\infty}
\end{align}
and
\begin{align}\label{Bmo}
\|Tf\|_{\mathrm{BMO}^{c}_{d}(\mathcal{N})}\lesssim\|f\|_{\infty}.
\end{align}

However, it suffices to estimate the column case-(\ref{Bmo}). Indeed, we assume (\ref{Bmo}). Using the fact that $\|g\|_{\mathrm{BMO}^{c}_{d}(\mathcal{N})}=\|g^{*}\|_{\mathrm{BMO}^{r}_{d}(\mathcal{N})}$ and taking the adjoint of both side in (\ref{Bmo1}), we then have
$$\|Tf\|_{\mathrm{BMO}^{r}_{d}(\mathcal{N})}=\|(Tf)^{*}\|_{\mathrm{BMO}^{c}_{d}(\mathcal{N})}=
\|Tf^{*}\|_{\mathrm{BMO}^{c}_{d}(\mathcal{N})}\lesssim\|f^{*}\|_{\infty}=\|f\|_{\infty},$$
which is the desired estimate (\ref{Bmo1}).

Furthermore, to estimate (\ref{Bmo}), we will use the following fact. Let $\alpha_Q$ be any operator
depending on $Q$ . If
$$\sup_{Q \in \Q} \Big\| \Big(
\frac{1}{|Q|} \int_Q \Big|Tf(x) -\alpha_Q\Big|^2 \, dx \Big)^{\frac12} \Big\|_{\M}\lesssim\|f\|_{\infty},$$
then (\ref{Bmo}) holds (see e.g. \cite{MP}). 

Now we are going to prove estimate (\ref{Bmo}).
\begin{proof}[proof of (\ref{Bmo}).]
Let $f\in L_{\infty}(\mathcal N)$ and $Q$ be a dyadic cube with center $c_Q$. Decompose $f$ as
$f=f\chi_{ 3Q}+f\chi_{\R^d\setminus 3Q}\triangleq f_1+f_2$.
If we set $\alpha=Tf_{2}(c_Q)$, then
$$Tf(x)-\alpha=Tf_{1}(x)+Tf_{2}(x)-Tf_{2}(c_Q).$$
Thus by the operator convexity inequality of $x\mapsto|x|^{2}$, 
$$\frac{1}{|Q|}\int_{Q}|Tf(x)-\alpha|^{2}dx\leq2(A+B),$$
where
\begin{eqnarray*}
A & = & \frac{1}{|Q|}\int_{Q}|Tf_{1}(x)|^{2}dx, \\
B & = & \frac{1}{|Q|}\int_{Q}|Tf_{2}(x)-Tf_{2}(c_Q)|^{2}dx.
\end{eqnarray*}
The first term $A$ is easy to handle. Indeed, by duality and Lemma \ref{t1},
\begin{align*}
\|A\|_{\M}& =\Big\|\frac{1}{|Q|}\int_{Q}|Tf_{1}(x)|^{2}dx\Big\|_{\M}\\
& \thickapprox\frac{1}{|Q|} \sup_{{\begin{subarray}{c}
\|a\|_{L_{1}(\M)}\leq1 \\ a\geq0
\end{subarray}}}\tau\Big(a\int_Q
|Tf_{1}(x)|^{2}dx\Big)\\
& =\frac{1}{|Q|} \sup_{{\begin{subarray}{c}
\|a\|_{L_{1}(\M)}\leq1 \\ a\geq0
\end{subarray}}}\tau\Big(\int_Q
|Tf_{1}(x)a^{\frac{1}{2}}|^{2}dx\Big)\\
& =\frac{1}{|Q|} \sup_{{\begin{subarray}{c}
\|a\|_{L_{1}(\M)}\leq1 \\ a\geq0
\end{subarray}}}\int_Q
\|Tf_{1}(x)a^{\frac{1}{2}}\|_{L_{2}(\M)}^{2}dx\\
& \leq \frac{1}{|Q|} \sup_{{\begin{subarray}{c}
\|a\|_{L_{1}(\M)}\leq1 \\ a\geq0
\end{subarray}}}\int_{\R^{d}}
\|Tf_{1}(x)a^{\frac{1}{2}}\|_{L_{2}(\M)}^{2}dx\\
& = \frac{1}{|Q|} \sup_{{\begin{subarray}{c}
\|a\|_{L_{1}(\M)}\leq1 \\ a\geq0
\end{subarray}}}\|T(f\chi_{3Q}a^{\frac{1}{2}})\|_{2}^{2}\\
&\lesssim\frac{1}{|Q|} \sup_{{\begin{subarray}{c}
\|a\|_{L_{1}(\M)}\leq1 \\ a\geq0
\end{subarray}}}\|f\chi_{3Q}a^{\frac{1}{2}}\|_{2}^{2}\lesssim\|f\|^{2}_{\infty},
\end{align*}
which is the desired estimate.

Now we turn to the second term $B$. Note that
\begin{align*}
Tf_{2}(x)-Tf_{2}(c_Q)
&\ =\sum_{k}\nu_{k}\big(M_{k}f_{2}(x)-\mathsf{E}_{k}f_{2}(x)-
(M_{k}f_{2}(c_{Q})-\mathsf{E}_{k}f_{2}(c_{Q}))\big)\\
&\ \triangleq\sum_{k}F_{k,Q}(x).
\end{align*}
For any $k$ satisfying $2^{-k}<\ell(Q)$, $F_{k,Q}(x)=0$ if $x\in Q$. Indeed, since $2^{-k}<\ell(Q)$ and $f_2$ is supported in $\R^d\setminus 3Q$, a simple geometric observation shows that both $\mathsf{E}_k f_2$ and $M_{k}f_2$ are supported in $\R^d\setminus Q$. Hence, in this case, for any $x\in Q$, $F_{k,Q}(x)=0$. On the other hand, for $k$ satisfying $2^{-k}\geq\ell(Q)$, we claim that for any $x\in Q$,
\begin{eqnarray}\label{97}
\|F_{k,Q}(x)\|_{\M}\lesssim2^{k}\ell(Q)\|f\|_{\infty}.
\end{eqnarray}
To see this, since $x$ and $c_Q$ are in the same cube in $\Q_{k}$, $\mathsf{E}_kf_2(x)=\mathsf{E}_k f_2(c_Q)$. Therefore,
\begin{align*}
\|F_{k,Q}(x)\|_{\M} &\leq\frac{1}{|B_{k}|} \Big\|\int_{x+B_{k}}f_2(y)dy -\int_{c_Q+B_{k}}f_2(y)dy\Big\|_{\M} \\
&\approx2^{kd} \Big\|\int_{\R^d}f_2(y)( \chi_{ (c_Q+B_{k})\setminus  (x+B_{k})}-\chi_{(x+B_{k})\setminus (c_Q+B_{k})})(y)dy\Big\|_{\M} \\
&\leq2^{kd}|  (c_Q+B_{k}) \Delta  (x+B_{k}) |\cdot \|f\| _{\infty}.
\end{align*}
Furthermore, the fact that $|(c_Q+B_{k}) \Delta (x+B_{k}) |\lesssim 2^{-k(d-1)}|x-c_Q|$ implies
$$\|F_{k,Q}(x)\|_{\M}\lesssim2^{k}|x-c_Q|\|f\|_{\infty}\lesssim 2^{k}\ell(Q)\|f\|_{\infty}.$$
This is precisely claim (\ref{97}).

Finally, with all these observations, we obtain
\begin{align*}
\|B\|_{\M}& =\Big\|\frac{1}{|Q|}\int_{Q}|\sum_{k}F_{k,Q}(x)|^{2}dx\Big\|_{\M}\\
& \leq\frac{1}{|Q|}\int_{Q}\Big(\sum_{k}\|F_{k,Q}(x)\|_{\M}\Big)^{2}\, dx\\
& \lesssim \ell(Q)^{2}\cdot  \|f\|^{2}_{\infty}\Big(\sum_{2^{-k} \geq\ell(Q)}2^{k}\Big)^{2}\lesssim \|f\|_\infty^{2}.
\end{align*}
This completes the $(L_{\infty},\mathrm{BMO})$ estimate.
\end{proof}
\section{Appendix A. Proof of Lemma \ref{t1}}
In this appendix, we prove Lemma \ref{t1}. As mentioned in the Introduction, the noncommutative $L_{2}$-boundedness of $T$ follows trivially from the corresponding commutative result. However, we provide a proof using similar idea as presented in \cite[Proposition 1]{JR02} (see also \cite[Lemma 3.13]{HX}) but not directly using the commutative result as a block box. It is this proof that inspires us to complete
the argument of the whole paper.

In the proof of Lemma \ref{t1}, we will need the following well-known result due to Cotlar, see e.g. \cite[Lemma 9.1]{Duo}.
\begin{Alemma}[Cotlar's Lemma]\label{Cotlar}\rm
Let $\mathcal{H}$ be a Hilbert space and let us consider a family
$(T_k)_{k \in \Z}$ of bounded operators on $\mathcal{H}$ with
finitely many non-zero $T_k$'s. Assume that there exists a
summable sequence $(\alpha_k)_{k \in \Z}$ such that
$$\max \Big\{ \big\| T_i^* T_j^{\null} \big\|_{\mathcal{B(H)}},
\big\| T_i^{\null} T_j^* \big\|_{\mathcal{B(H)}} \Big\} \, \le \,
\alpha_{i-j}^2$$ for all $i,j \in \Z$. Then 
$$\Big\| \sum_k T_k \Big\|_{\mathcal{B(H)}} \, \le \, \sum_k
\alpha_k.$$
\end{Alemma}
Indeed, we will actually apply a variant of Cotlar's Lemma obtained by Carbery \cite{carbery}.
\begin{Alemma}[Carbery]\label{Carbery}\rm
Let $(Q_{k})_{k\in\Z}$ and $(R_{k})_{k\in\Z}$ be two sequences of operators on a Hilbert space $\mathcal {H}$. For some $\varepsilon>0$, 
assume that operators $(Q_{k})_{k\in\Z}$ and $(R_{k})_{k\in\Z}$ satisfy the following properties:
\begin{enumerate}[(1)]
	\item [(i)] $\sum_{j}Q_{j}=I$ (in the sense of strong operator topology);
	\item[(ii)] for each $j,k$, $\|Q_{j}^{\ast}Q_{k}\|_{\mathcal{B(H)}}+\|Q_{j}Q_{k}^{\ast}\|_{\mathcal{B(H)}}\lesssim 2^{-\varepsilon|j-k|}$;
	
	\item[(iii)] for each $j,k$, $\|R_{j}Q_{k}^{\ast}\|_{\mathcal{B(H)}}+\|R_{j}^{\ast}Q_{k}\|_{\mathcal{B(H)}}\lesssim 2^{-\varepsilon|j-k|}$;
	
	\item[(iv)] for each $j$, $\|R_{j}\|_{\mathcal{B(H)}}\lesssim1$ .
	\end{enumerate}
Then $(R_{j})$ satisfy the hypothesis of Cotlar's Lemma and so $\sum_j R_j$ is a bounded operator on $\mathcal {H}$.
\end{Alemma}
In the following, we give one more fact. Let $n\in\mathbb Z$ and $B\subset\R^{d}$ be a Euclidean ball. 
For any integer $k<n$ and $1\leq p\leq\infty$, we define a linear operator $M_{k,n}: L_p(\mathcal N)\rightarrow L_p(\mathcal N)$ as
\begin{align}\label{mkn}
M_{k,n}u(x)=\frac{1}{|B_k|}\int_{\mathcal{I}(x+B_k,n)}u(y)dy,
\end{align}
where
\begin{eqnarray}\label{C18}
	\mathcal{I}(B,n)=\bigcup_{{\begin{subarray}{c}
				Q\in\Q_{n} \\ \partial B\cap Q\neq \emptyset
	\end{subarray}}} B\cap Q.
\end{eqnarray}
The following property of $M_{k,n}$ is obtained in \cite[Lemma 3.8]{HX}.
\begin{Alemma}\label{kn}\rm
Let $M_{k,n}$ be the linear operator defined as (\ref{mkn}) with $k<n$. Then for $1\leq p\leq\infty$ and $u\in L_p(\mathcal N)$, 
$$\|M_{k,n}u\|_{p}\lesssim 2^{k-n}\|u\|_p.$$
\end{Alemma}

Now we are ready to prove Lemma \ref{t1}.
\begin{proof}[Proof of Lemma \ref{t1}.]
Without loss of generality,
we may assume that $h$ is positive. According to Lemma \ref{Carbery}, we will take $R_{k}=\nu_{k}(M_{k}-\mathsf{E}_{k})$. It is clear that $\|R_{k}\|_{\mathcal{B}(L_{2}(\mathcal{N}))}\lesssim 1$. On the other hand, let $Q_{n}=\Delta_{n}$, where $\Delta_{n}=\mathsf{E}_{n}-\mathsf{E}_{n-1}$ is the martingale difference. With this choice of $Q_{n}$, it can be readily seen that $Q_{n}$ satisfies assumptions (i) and (ii) in Lemma \ref{Carbery}. Therefore, all that remains is to check assumption (iii) in Lemma \ref{Carbery}.

We first prove
\begin{align}\label{23333}
\|R_{k}Q_{n}h\|^{2}_{2}\lesssim 2^{-|n-k|}\|Q_{n}h\|^{2}_{2}.
\end{align}
Recall that $Q_{n}h=dh_{n}$ and $\|(\nu_{k})\|_{\ell_{\infty}}\leq1$. Then
$$\|R_{k}Q_{n}h\|^{2}_{2}=\|\nu_{k}(M_{k}-\mathsf{E}_{k})dh_{n}\|^{2}_{2}
\leq\|(M_{k}-\mathsf{E}_{k})dh_{n}\|^{2}_{2}.$$
Therefore, it suffices to show
\begin{align}\label{3}
\|(M_{k}-\mathsf{E}_k)dh_{n} \|^{2}_{2}\lesssim 2^{-|n-k|}\|dh_{n} \|^{2}_{2}.
\end{align}
Consider (\ref{3}) in the case $k\geq n$. Note that in this case $\mathsf{E}_kdh_{n}=dh_{n}$. Thus it is enough to show
\begin{align}\label{4}
\|M_{k}dh_{n}-dh_n\|^{2}_{2}\lesssim 2^{n-k}\|dh_n\|^{2}_{2}.
\end{align}
To this end, we write
\begin{align*}
    \|M_{k}dh_n-dh_n\|^{2}_{2}
    =&\ \int_{\R^d}\|M_{k}dh_n(x)-dh_n(x)\|^{2}_{L_{2}(\M)}dx\\
    =&\ \sum_{H\in\Q_{n}}\int_H\|M_{k}dh_n(x)-dh_n(x)\|^{2}_{L_{2}(\M)}dx.
\end{align*}
Notice that $dh_{n}$ is a constant operator on $H\in\Q_{n}$. Then $M_{k}dh_n(x)-dh_n(x)=0$ if $x+ B_{k}\subset H$. Thus for $x\in H$, $(M_{k}dh_n-dh_n)(x)$ may be nonzero only if $x+B_{k}$ intersects with the complement of $H$---$H^c$. On the other hand, given a atom $H$ and a Euclidean ball $B$ in $\R^{d}$, if we define
$$\mathcal{H}(B,H)= \{x\in H|(x+B)\cap H^c\neq \emptyset\}.$$
then it is easy to verify that for a fixed $H\in\Q_{n}$
  $$|\mathcal{H} (B_{k},H)|\lesssim2^{(d-1)(-n)}\cdot 2^{-k}.$$
By these observations, we see that  
\begin{align}\label{10}
	\int_H \|M_{k}dh_n(x)-dh_n(x)\|^{2}_{L_{2}(\M)}dx\lesssim2^{(d-1)(-n)}\cdot2^{-k}\cdot m_H^{2}\lesssim2^{n-k} \int_Hm_H^{2}dx,
\end{align}  
where $m_H$ is the maximum of $\|dh_n(x)\|_{L_{2}(\M)}$ on $H$ and the cubes in $\Q_{n}$ neighboring $H$. Since $m_H$ is a constant on $H$, $\int_{\R^d}m^{2}_H \lesssim\int_{\R^d}\|dh_n(x)\|^{2}_{L_{2}(\M)}dx$. Finally, summing over all $H\in\Q_{n}$ in (\ref{10}), we obtain the desired estimate (\ref{4}).

Now we estimate  (\ref{3}) in the case $n> k$. Note that $\mathsf{E}_kdh_n=0$ for $n> k$. Hence, it suffices to prove
\begin{align}\label{5}
\|M_{k}dh_n\|^{2}_2\lesssim2^{k-n}\|dh_n\|^{2}_{2}.
\end{align}
The cancellation property of $dh_n$ over atoms in $\Q_{n-1}$ gives $M_{k}dh_n=M_{k,n-1}dh_n$. Therefore, (\ref{5}) follows directly from Lemma \ref{kn}. 

Next we consider
\begin{align}\label{233333}
\|R_{k}^{\ast}Q_{n}h\|^{2}_{2}=\|\nu_{k}(M_{k}^{\ast}-\mathsf{E}_{k})dh_{n}\|^{2}_{2}\lesssim 2^{-|n-k|}\|Q_{n}h\|^{2}_{2}.
\end{align}
Note that the adjoint operator $M^{*}_{k}$ of $M_{k}$, defined by
$$\langle M_{k}f,g\rangle=\langle f,M^{*}_{k}g\rangle$$
associated with the kernel that coincide with the function $|B_{k}|^{-1}\chi_{B_{k}}(-x)$. Since $B_{k}$ is symmetric, $M_{k}^{*}=M_{k}$. Hence, by  (\ref{23333}),
$\|R_{k}^{\ast}Q_{n}\|_{\mathcal{B}(L_{2}(\mathcal{N}))}\lesssim 2^{-|n-k|}$. Therefore, $R_{k}$ satisfies the hypothesis of Lemma \ref{Carbery}.
Consequently, $\|T\|_{\mathcal{B}(L_{2}(\mathcal{N}))}=\|\sum_{k}R_{k}\|_{\mathcal{B}(L_{2}(\mathcal{N}))}\lesssim1$. This finishes the proof.
\end{proof}
\section{Appendix B. Some further problems}
In this section, we list some further problems. Inspired by the argument in \cite{Christ, Seeger}, we would like to seek a new proof of weak type $(1,1)$ estimate for the bad function based on the $L_{2}$-norm method.


Let $f \in \mathcal{N}_{c,+}$ and $\lambda>0$. If we set
\begin{equation}\label{czd2343}
\begin{array}{rclcrcl} \displaystyle b_d & = & \displaystyle \sum_{n\in\Z} p_n \hskip1pt (f - f_n) \hskip1pt p_n, & \quad &
b_{\mathit{off}} & = & \displaystyle \sum_{n\in\Z} p_n (f-f_{n}) q_n+q_n (f-f_{n}) p_n,
\end{array}
\end{equation}
then it is easy to show that $b=b_d+b_{\mathit{off}}$, where $b_d$ (resp. $b_{\mathit{off}}$) is called the diagonal (resp. off-diagonal) part of $b$. 
\begin{Bproposition}\rm
The following estimate holds.
\begin{align}\label{HLX7787}
\varphi(\chi_{(\lambda,\infty)}(|Tb_{d}|))\lesssim\frac{\| f\|_{1}}{\lambda}.
\end{align}
\end{Bproposition}
In this section, we intend to give another proof by applying the $L_{2}$-norm method.
By similar reasoning, we are reduced to showing
\begin{align}\label{bd334}
\varphi \big(\chi_{(\lambda/4,\infty)}(|\zeta Tb_{d}\zeta|)\big)\lesssim\frac{ \|f\|_{1}}{\lambda},
\end{align}
where $\zeta$ was defined in (\ref{czd9}).
\begin{proof}[Proof of estimate \eqref{bd334}.]
Recall that since we can assume that 	$m_{\lambda}(f)=0$, $p_n=0$ for all $n\leq 0$.  Thus $b_{d}=\sum_{n=1}^{\infty}b_{d,n}$, where $b_{d,n}=p_n (f-f_{n}) p_n$ as in (\ref{czd2343}). By the property of $\zeta$ (see e.g. \cite[Lemma 2.1]{CCP}) and 
using the same argument as in (\ref{diff}), we  have
\begin{align*}
	\zeta Tb_{d}\zeta
	=&\ \zeta\sum_{n=1}^{\infty}\sum_{k:k<n}\nu_{k}M_{k}b_{d,n}\zeta=\zeta\sum_{n=1}^{\infty}\sum_{j=1}^{\infty}\nu_{n-j}M_{n-j}b_{d,n}\zeta\\
	=&\ \zeta\sum_{n=1}^{\infty}\sum_{j=1-n}^{\infty}\nu_{j}M_{j}b_{d,n+j}\zeta=\zeta\sum_{n=1}^{\infty}\sum_{j\in\Z}\nu_{j}M_{j}b_{d,n+j}\zeta,
\end{align*}
where the third equation is based on the Fubini theorem and the last equation follows from the fact that $p_n=0$ for all $n\leq 0$.

For convenience, in the following, we may assume $\nu_{j}=1$ for all $j$ since $(\nu_{j})$ is just a bounded sequence. Then by the Chebychev and Minkowski inequalities, we find that
$$\varphi \big(\chi_{(\lambda/4,\infty)}(|\zeta Tb_{d}\zeta|)\big)\lesssim\frac{\big(\sum_{n=1}^{\infty}\|\sum_{j\in\Z}M_{j}b_{d,n+j}\|_{2}\big)^{2}}{\lambda^{2}}.$$
Thus it is enough to show
\begin{equation}\label{czd23443}
\|\sum_{j\in\Z}M_{j}b_{d,n+j}\|_{2}^{2}\lesssim2^{-n}\lambda\|f\|_{1}.
\end{equation}
Notice that for each $j\in\Z$, $\mathsf{E}_{n+j}b_{d,n+j}=0$. Therefore,
\begin{equation*}\label{czd234443}
M_{j}b_{d,n+j}=M_{j,n+j}b_{d,n+j},
\end{equation*}
where $M_{j,n+j}$ was defined in (\ref{mkn}). Hence,
$$\big\|\sum_{j\in\Z}M_{j}b_{d,n+j}\big\|_{2}^{2}=\big\|\sum_{j\in\Z}M_{j,n+j}b_{d,n+j}\big\|_{2}^{2}.$$
Then we use the Minkowski inequality and the definition of $b_{d,n+j}$ to get
\begin{align*}
    \big\|\sum_{j\in\Z}M_{j,n+j}b_{d,n+j}\big\|_{2}^{2}
    \lesssim&\ \big\|\sum_{j\in\Z}M_{j,n+j}p_{n+j}f_{n+j}p_{n+j}\big\|_{2}^{2}+\big\|\sum_{j\in\Z}M_{j,n+j}p_{n+j}fp_{n+j}\big\|_{2}^{2}\\
    \triangleq&\ \|F_{1,n}\|_{2}^{2}+\|F_{2,n}\|_{2}^{2}.
\end{align*}
Since $M_{j,n+j}p_{n+j}f_{n+j}p_{n+j}$ is a positive operator in $\mathcal{N}$,
\begin{align*}
	M_{j,n+j}p_{n+j}f_{n+j}p_{n+j}(x)
	\leq&\ \frac{1}{|B_j|}\sum_{{\begin{subarray}{c}
				Q\in\Q_{n+j} \\ (x+\partial B_j)\cap Q\neq \emptyset
	\end{subarray}}} \int_{Q}p_{n+j}f_{n+j}p_{n+j}(y)dy\\
	\leq&\ \widetilde{M}_{j,n+j}p_{n+j}f_{n+j}p_{n+j}(x)
\end{align*}
where $\widetilde{M}_{j,n+j}$ is defined by
\begin{equation}\label{czd234443}
	\widetilde{M}_{j,n+j}p_{n+j}f_{n+j}p_{n+j}=\frac{1}{|B_{j}|}\chi_{I_{j,n+j}}\ast p_{n+j}f_{n+j}p_{n+j}.
\end{equation}
and
$$I_{j,n+j}\triangleq\partial B_j+\sqrt{d}B_{n+j}.$$
Moreover, a simple calculation shows that $|I_{j,n+j}|\lesssim2^{-j(d-1)}2^{-n-j}$.


Similarly, by noting $\mathsf{E}_{n+j}p_{n+j}fp_{n+j}=p_{n+j}f_{n+j}p_{n+j}$, we conclude that
$$M_{j,n+j}p_{n+j}fp_{n+j}\leq\widetilde{M}_{j,n+j}p_{n+j}f_{n+j}p_{n+j}.$$
As a consequence,
$$\max\Big\{\|F_{1,n}\|_{2}^{2},\|F_{2,n}\|_{2}^{2}\Big\}\leq\big\|\sum_{j\in\Z}\widetilde{M}_{j,n+j}p_{n+j}f_{n+j}p_{n+j}\big\|_{2}^{2}.$$
Therefore, it suffices to show
$$\big\|\sum_{j\in\Z}\widetilde{M}_{j,n+j}p_{n+j}f_{n+j}p_{n+j}\big\|_{2}^{2}\lesssim2^{-n}\lambda\|f\|_{1}.$$
Note that $\widetilde{M}_{j,n+j}p_{n+j}f_{n+j}p_{n+j}$ is still a positive operator in $\mathcal{N}$. Then by the definition of $I_{j,n+j}$,
\begin{align*}
&\quad\big\|\sum_{j\in\Z}\widetilde{M}_{j,n+j}p_{n+j}f_{n+j}p_{n+j}\big\|_{2}^{2}\\
 & =\varphi\Big(\big(\sum_{j\in\Z}\widetilde{M}_{j,n+j}p_{n+j}f_{n+j}p_{n+j}\big)
    \big(\sum_{i\in\Z}\widetilde{M}_{i,n+i}p_{n+i}f_{n+i}p_{n+i}\big)\Big)\\& \lesssim\sum_{j\in\Z}\sum_{i:i\geq j}\varphi\Big(\big(\widetilde{M}_{j,n+j}p_{n+j}f_{n+j}p_{n+j}\big)
    \big(\widetilde{M}_{i,n+i}p_{n+i}f_{n+i}p_{n+i}\big)\Big)\\
 & =\sum_{j\in\Z}\sum_{i:i\geq j}\varphi\Big(\widetilde{M}_{j,n+j}\circ \widetilde{M}_{i,n+i}(p_{n+i}f_{n+i}p_{n+i})(p_{n+j}f_{n+j}p_{n+j})\Big)\\
 & \lesssim\sum_{j\in\Z}\sum_{i:i\geq j}\lambda\varphi\Big(\widetilde{M}_{j,n+j}\circ \widetilde{M}_{i,n+i}(p_{n+i})(p_{n+j}f_{n+j}p_{n+j})\Big),
\end{align*}
where the last inequality follows from the fact that $p_{n+i}f_{n+i}p_{n+i}\lesssim\lambda p_{n+i}$.

In the following, we claim that to complete the argument, it suffices to show
\begin{equation}\label{czd2323234443}
\sum_{i:i\geq j}\widetilde{M}_{j,n+j}\circ \widetilde{M}_{i,n+i}(p_{n+i})\lesssim2^{-n}.
\end{equation}
Indeed, by (\ref{czd2323234443}), we get
$$\big\|\sum_{j\in\Z}\widetilde{M}_{j,n+j}p_{n+j}f_{n+j}p_{n+j}\big\|_{2}^{2}\lesssim\sum_{j\in\Z}\lambda2^{-n}\varphi\big( (p_{n+j}f_{n+j}p_{n+j})\big)\leq2^{-n}\lambda\|f\|_{1},$$
which is exactly the desired estimate.

To show (\ref{czd2323234443}),
we see that for any $i\geq j$, $\chi_{I_{j,n+j}}\ast \chi_{I_{i,n+i}}$ is supported in
\begin{equation}\label{czd2344433}
\overline{I}_{j,n+j}\triangleq I_{j,n+j}+I_{j,n+j},
\end{equation}
and $|\overline{I}_{j,n+j}|\approx2^{-jd}$.
Moreover, by the Young inequality,
\begin{equation}\label{czd23444323}
\|\chi_{I_{j,n+j}}\ast \chi_{I_{i,n+i}}\|_{\infty}\leq\|\chi_{I_{j,n+j}}\|_{\infty}
\|\chi_{I_{i,n+i}}\|_{1}\lesssim2^{-i(d-1)}2^{-i-n}.
\end{equation}
Putting these estimates together and using property (\ref{czd1123}), we finally conclude that for any fixed $x$,
\begin{align*}
    \sum_{i:i\geq j}\widetilde{M}_{j,n+j}\circ \widetilde{M}_{i,n+i}(p_{n+i})(x)
    \lesssim&\ \sum_{i:i\geq j}2^{id}2^{jd}2^{-i(d-1)}2^{-i-n}\int_{x+\overline{I}_{j,n+j}} p_{n+i}(y)dy\\
    =&\ 2^{jd}2^{-n}\int_{x+\overline{I}_{j,n+j}}\sum_{i:i\geq j} p_{n+i}(y)dy\\
    \leq& 2^{jd}2^{-n}|x+\overline{I}_{j,n+j}|\approx2^{-n},
\end{align*}
where in the third inequality, we applied the fundamental fact $\sum_{j\in\Z}p_{j}\leq1_{\mathcal{N}}$.
This completes the proof.
\end{proof}
\begin{Bremark}\rm
However, with a moment's thought, we have no idea that how to use $L_{2}$-norm method to prove the weak type $(1,1)$ behavior related to the off-diagonal part $b_{\mathit{off}}$ defined in (\ref{czd2343}). To our knowledge, the $L_{2}$-norm method should be more powerful and it is expected to open the research on the singular integral operators of rough kernels \cite{Christ, Seeger} and other related problems in noncommutative Calder\'on-Zygmund theory.
\end{Bremark}

\noindent {\bf Acknowledgements} \
I would like to thank Guixiang Hong for many valuable discussions,
his guidance throughout the making of this paper,
\'{E}ric Ricard and L\'{e}onard Cadilhac for communicating to us the present version of Calder\'on-Zygmund decomposition-Theorem \ref{czdecom}, Dr. Liang Wang for the useful discussion. In particular, I thank L\'{e}onard Cadilhac for providing us a note with a new proof of the weak type $(1,1)$ estimate of Calder\'on-Zygmund operator \cite{C2}.
I am also very grateful to the referees for their very careful reading and valuable comments. 

This work was supported by Natural Science Foundation of China Grant 12071355, the Basic Science Research Program through the National Research Foundation of Korea (NRF) Grant NRF-2017R1E1A1A03070510 and the Samsung Science and Technology Foundation under Project Number SSTF-BA2002-01.

\end{document}